\documentclass[pdflatex,sn-mathphys-num]{sn-jnl}
\usepackage{amsmath,amssymb,amsfonts,amsthm,mathtools}
\usepackage[T1]{fontenc}
\usepackage{lmodern}
\usepackage{microtype}
\hypersetup{
  pdftitle={A Measure-Theoretic Derivation of Zadeh Fuzzy Logic},
  pdfauthor={Angshul Majumdar},
  pdfsubject={Derivation of Zadeh fuzzy logic from measure theory},
  pdfkeywords={Zadeh fuzzy logic, measure algebra, Boolean algebra, hyperoperation, t-norm, MV-algebra}
}
\newtheorem{theorem}{Theorem}
\newtheorem{proposition}{Proposition}
\newtheorem{corollary}{Corollary}

\newcommand{\cB}{\mathcal{B}}
\newcommand{\cE}{\mathcal{E}}
\newcommand{\TL}{T_{\mathrm L}}
\newcommand{\TG}{T_{\mathrm G}}
\newcommand{\TP}{T_{\mathrm P}}
\newcommand{\st}{\mathbin{\star}}
\newcommand{\boxp}{\mathbin{\boxplus}}
\newcommand{\setdiff}{\mathbin{\backslash}}
\newcommand{\Luk}{\L{}ukasiewicz}

\begin{document}

\title[A Measure-Theoretic Derivation of Zadeh Fuzzy Logic]{A Measure-Theoretic Derivation of Zadeh Fuzzy Logic}

\author*[1]{\fnm{Angshul} \sur{Majumdar}}\email{angshul@iiitd.ac.in}

\affil*[1]{\orgdiv{Department of Electronics and Communication Engineering},
\orgname{Indraprastha Institute of Information Technology Delhi},
\orgaddress{\street{Okhla Industrial Estate, Phase III}, \city{New Delhi},
\postcode{110020}, \country{India}}}

\abstract{Zadeh fuzzy logic is normally introduced by assigning scalar membership degrees and postulating standard complement, minimum conjunction, and maximum disjunction. We derive this calculus from a strictly positive atomless normalized Boolean measure algebra. Equality of measure sends Boolean complement uniquely to $1-x$, but it cannot send meet to a single-valued operation: the exact scalar image of conjunction is the overlap interval $[(x+y-1)^+,\min\{x,y\}]$. We quantify the information lost under scalarization, prove that no nontrivial measure-preserving compression can retain meet, and show that the overlap operation is associative as a hyperoperation. Zadeh conjunction $\min\{x,y\}$ is then forced as the unique nondecreasing idempotent selection and is realized globally by nested event chains; its De Morgan dual is Zadeh disjunction $\max\{x,y\}$. Thus the min--max--standard-negation calculus follows from measure theory once the nesting commitment is made explicit. The \Luk{} and product conjunctions arise as minimum-overlap and independent-factor selections, clarifying alternative scalarizations. Exact minimax results, an ordered-valued extension, and a three-dimensional realization obstruction delimit what scalar fuzzy logic retains from Boolean structure.}


\maketitle

\section{Introduction}\label{sec:intro}
Zadeh's original fuzzy-set calculus assigns each proposition or set a degree in $[0,1]$ and uses the standard operations $1-x$, $\min\{x,y\}$, and $\max\{x,y\}$ for complement, conjunction, and disjunction \cite{Zadeh1965}. Later theories place these operations within the broader classes of t-norms, residua, and many-valued calculi \cite{Klement2000,Hajek1998}. Those frameworks characterize admissible scalar operations once scalar truth degrees have already been adopted. The question here runs in the opposite direction: can the Zadeh operations themselves be derived from a Boolean measure-theoretic substrate rather than postulated on $[0,1]$?

Algebraic logic provides the natural setting for this question because it studies logical operations through their algebraic semantics \cite{FontJansanaPigozzi2003}. MV-, BL-, and MTL-algebras organize major families of many-valued logics and their truth-functional conjunctions \cite{Panti1999,CignoliTorrens2006,NogueraEstevaGispert2005,Turunen2025,Rump2026}. Quantitative comparisons further distinguish the calculi generated by different fuzzy conjunctions \cite{KostrzyckaZaionc2024}. Our aim is not another axiomatization of those operations. It is to identify the exact scalar structure produced when Boolean events are quotiented by equality of measure and then to determine which additional structural requirement selects Zadeh's min--max calculus.

The relation between probability and logic is often formulated by enriching a language with probability terms or quantifiers over events \cite{Speranski2025}. We instead compress the event algebra itself to marginal measures. This makes the problem representational: Boolean complement, meet, join, and order are examined after all distinctions between equal-measure events have been erased. Such representation questions are central throughout non-classical algebraic logic \cite{Umadevi2025}, but the specific quotient considered here exposes a sharp asymmetry between complement and conjunction.

The mathematical problem can be stated directly. Let $(\cB,\mu)$ be an atomless normalized Boolean measure algebra and write $a\equiv_\mu b$ iff $\mu(a)=\mu(b)$. The exact scalar image of meet is
\begin{equation}
 x\st y:=\{\mu(a\wedge b):\mu(a)=x,\ \mu(b)=y\}.
 \label{eq:star-def}
\end{equation}
Classical Fr\'echet theory evaluates this set. Our first task is to determine the algebraic structure carried by these attainable sets and to quantify the Boolean information destroyed by the quotient. Our second, central task is to recover Zadeh fuzzy logic. Complement descends uniquely to $1-x$. Meet becomes an associative overlap hyperoperation, and imposing the idempotence and monotonicity inherited by Zadeh conjunction forces the upper-overlap selection $\min\{x,y\}$. De Morgan duality then gives $\max\{x,y\}$, while nested event chains realize both operations simultaneously for all degrees. In this precise sense, the Zadeh min--max--standard-negation calculus is derived from measure theory.

The remaining results locate that derivation within a wider landscape. The lower-overlap and independent-factor selections yield the \Luk{} and product conjunctions. An exact intersection-test discrepancy gives a no-compression theorem for Boolean meet. The overlap hyperoperation has an exact finite-arity spectrum; product attains the sharp binary minimax error but ceases to be minimax beyond arity two. The construction extends to interval-divisible measures valued in unital Abelian $\ell$-groups, and a three-dimensional obstruction separates pointwise sharp attainability from simultaneous realization.

Section~\ref{sec:scalarization} determines the exact loss under scalarization. Section~\ref{sec:hyper} derives the Zadeh operations and analyzes alternative selections. Section~\ref{sec:ordered} gives the ordered-valued extension, Section~\ref{sec:global} separates local attainability from global realization, and Section~\ref{sec:conclusion} concludes.

The argument uses classical results only at clearly identified points: nonatomic divisibility, sharp Fr\'echet bounds and their attainability, product-measure multiplicativity, and standard copula representation \cite{Halmos1950,Sierpinski1922,Frechet1935,Boole1854,Sklar1959,Nelsen2006}. Standard facts about many-valued calculi, residuation, the Frank equation, continuous t-norms, and the $\ell$-group representation of MV-algebras are likewise used in their established form \cite{Chang1958,Godel1932,Dummett1959,Frank1979,Mostert1957,Ling1965,Mundici1986,Cignoli2000}. The contribution is the derivation of Zadeh's operations and the exact analysis of what is lost, selected, and retained when Boolean measure structure is scalarized.

\section{Scalarization and Exact Information Loss}\label{sec:scalarization}
Throughout, $(\cB,\mu)$ is a strictly positive atomless normalized measure algebra. We use the classical divisibility theorem: if $0\le t\le\mu(a)$, some $b\le a$ satisfies $\mu(b)=t$ \cite{Halmos1950,Sierpinski1922}. Also $\mu(a\vee b)+\mu(a\wedge b)=\mu(a)+\mu(b)$ and $\mu(a^c)=1-\mu(a)$. We write $a\setdiff b=a\wedge b^c$ and $a\triangle b=(a\setdiff b)\vee(b\setdiff a)$.

\begin{proposition}[Scalar complement is uniquely determined]\label{prop:complement}
If $N:[0,1]\to[0,1]$ satisfies $\mu(a^c)=N(\mu(a))$ for every $a\in\cB$, then $N(x)=1-x$.
\end{proposition}
\begin{proof}
For each $x\in[0,1]$, divisibility gives $a$ with $\mu(a)=x$. Then $N(x)=\mu(a^c)=1-\mu(a)=1-x$.
\end{proof}
Thus complement genuinely descends before any conjunction is selected.

The next formula is a direct measure-algebra specialization of the classical Hahn--Jordan decomposition for signed measures \cite{Halmos1950,Bogachev2007}; we include the short proof because its compression consequence is central here.

\begin{proposition}[Exact intersection-test discrepancy]\label{prop:discrepancy}
For $a,b\in\cB$, define $D(a,b):=\sup_{c\in\cB}|\mu(a\wedge c)-\mu(b\wedge c)|$. Then
\begin{equation}
D(a,b)=\max\{\mu(a\setdiff b),\mu(b\setdiff a)\}=\frac{\mu(a\triangle b)+|\mu(a)-\mu(b)|}{2}.
\label{eq:discrepancy}
\end{equation}
In particular, if $\mu(a)=\mu(b)$, then $D(a,b)=\mu(a\triangle b)/2$.
\end{proposition}
\begin{proof}
Put $p=\mu(a\setdiff b)$ and $q=\mu(b\setdiff a)$. The common part cancels, so for every $c$ the difference $\mu(a\wedge c)-\mu(b\wedge c)$ lies in $[-q,p]$. Taking $c=a\setdiff b$ or $c=b\setdiff a$ attains the two endpoints. Finally $\mu(a\triangle b)=p+q$ and $\mu(a)-\mu(b)=p-q$, which yields \eqref{eq:discrepancy}.
\end{proof}
The same formula shows that $D$ is a metric and $\frac12\mu(a\triangle b)\le D(a,b)\le\mu(a\triangle b)$, so it is topologically equivalent to the usual symmetric-difference metric.

\begin{corollary}[No nontrivial meet-preserving measure compression]\label{cor:no-compression}
Let $\sim$ be an equivalence relation on $\cB$ such that $a\sim b$ implies $\mu(a)=\mu(b)$ and $a\wedge c\sim b\wedge c$ for every $c\in\cB$. Then $a\sim b$ implies $a=b$.
\end{corollary}
\begin{proof}
If $a\sim b$, meet compatibility and measure preservation give $\mu(a\wedge c)=\mu(b\wedge c)$ for every $c$. Thus $D(a,b)=0$ by Proposition~\ref{prop:discrepancy}. Hence $\mu(a\setdiff b)=\mu(b\setdiff a)=0$, and strict positivity gives $a=b$.
\end{proof}

\begin{corollary}[Diameter of a scalar fiber]\label{cor:fiber-diameter}
For fixed $x\in[0,1]$, $\sup\{D(a,b):\mu(a)=\mu(b)=x\}=\min\{x,1-x\}$.
\end{corollary}
\begin{proof}
For equal measures, Proposition~\ref{prop:discrepancy} gives $D(a,b)=x-\mu(a\wedge b)$. The sharp Fr\'echet lower bound gives $\mu(a\wedge b)\ge(2x-1)^+$, hence $D(a,b)\le x-(2x-1)^+=\min\{x,1-x\}$. Sharp attainability realizes equality.
\end{proof}
Thus equality of measure is not merely ``not enough dependence information'': every nontrivial measure-preserving compression destroys meet, and Corollary~\ref{cor:fiber-diameter} quantifies the diameter of what is collapsed to one scalar. Complement is exceptional because $\mu(a^c)=1-\mu(a)$.

\begin{proposition}[What survives of order]\label{prop:order}
For $x,y\in[0,1]$, $x\le y$ iff there exist $a,b\in\cB$ with $\mu(a)=x$, $\mu(b)=y$, and $a\le b$. However, the measures of particular representatives do not determine whether $a\le b$.
\end{proposition}
\begin{proof}
If $a\le b$, monotonicity gives $x\le y$. Conversely, if $x\le y$, choose $b$ with $\mu(b)=y$ and use divisibility inside $b$ to choose $a\le b$ with $\mu(a)=x$. For nonrecoverability, choose $0<x<1/2$ and disjoint $a,b$ with $\mu(a)=\mu(b)=x$; neither is below the other. By contrast, if $a\le b$ and $\mu(a)=\mu(b)$, then $\mu(b\setdiff a)=0$, so strict positivity forces $a=b$.
\end{proof}

The classical sharp overlap theorem now gives
\begin{equation}
 x\st y=[(x+y-1)^+,\min\{x,y\}],
 \label{eq:frechet}
\end{equation}
with every value attained \cite{Frechet1935,Boole1854}. Define likewise $x\boxp y:=\{\mu(a\vee b):\mu(a)=x,\ \mu(b)=y\}$. Its exact value is $[\max\{x,y\},\min\{1,x+y\}]$, and Boolean complement induces the hyper-De Morgan identity $x\boxp y=1-((1-x)\st(1-y))$. The width of the overlap interval is $\min\{x,y,1-x,1-y\}$, so it vanishes exactly when one marginal is crisp.

\begin{proposition}[Exact dependence coordinate]\label{prop:dependence}
For events of marginal measures $x,y$, put $\kappa(a,b)=\mu(a\wedge b)-xy$. Its attainable range is
$[-\min\{xy,(1-x)(1-y)\},\min\{x(1-y),y(1-x)\}]$. Thus product conjunction is precisely the zero-dependence point inside every Fr\'echet fiber.
\end{proposition}
\begin{proof}
Subtract $xy$ from the endpoints in \eqref{eq:frechet}. If $x+y\le1$ the lower endpoint is $-xy$, while if $x+y\ge1$ it is $-(1-x)(1-y)$. For the upper endpoint, if $x\le y$ then $\min\{x,y\}-xy=x(1-y)\le y(1-x)$; the other case is symmetric. Since $xy$ lies in \eqref{eq:frechet}, zero is attainable, canonically by independent product factors.
\end{proof}

\begin{proposition}[Uniform realization against a fixed representative]\label{prop:fixed-realization}
Fix $a\in\cB$ with $\mu(a)=x$. For every $y\in[0,1]$ and every $z\in x\st y$, there exists $b\in\cB$ with $\mu(b)=y$ and $\mu(a\wedge b)=z$.
\end{proposition}
\begin{proof}
By \eqref{eq:frechet}, $0\le z\le x$ and $0\le y-z\le1-x=\mu(a^c)$. Divisibility gives $p\le a$ with $\mu(p)=z$ and $q\le a^c$ with $\mu(q)=y-z$. For $b=p\vee q$, disjoint additivity gives $\mu(b)=y$, while $a\wedge b=p$.
\end{proof}
By definition, any set-valued scalar rule $H$ sound for every Boolean meet must contain every attainable value, hence $x\st y\subseteq H(x,y)$. Thus $\st$ is the unique smallest sound marginal-only abstraction of meet. Equation~\eqref{eq:frechet} is exact only when the two marginal measures are the constraints: for example, $\mu(a\wedge a^c)=0$ although $x\st(1-x)$ is generally an interval, and $\mu(a\wedge a)=x$ although $x\st x$ is generally non-singleton.

\section{Derivation of Zadeh Conjunction and Alternative Selections}\label{sec:hyper}
For nonempty $A\subseteq[0,1]$ define $A\st z=\bigcup_{a\in A}(a\st z)$ and similarly on the right.

\begin{theorem}[Hyperassociativity]\label{thm:hyperassoc}
For all $x,y,z\in[0,1]$,
\begin{equation}
(x\st y)\st z=x\st(y\st z)=[(x+y+z-2)^+,\min\{x,y,z\}].
\label{eq:hyperassoc}
\end{equation}
\end{theorem}
\begin{proof}
Let $w\in(x\st y)\st z$. Then some $v\in x\st y$ satisfies $w\in v\st z$. Hence $v\ge x+y-1$, $w\ge v+z-1$, and $w\le v,z$, so $w\ge(x+y+z-2)^+$ and $w\le\min\{x,y,z\}$. Conversely assume these bounds and put $v=\max\{w,x+y-1,0\}$. Then $(x+y-1)^+\le v\le\min\{x,y\}$, so $v\in x\st y$. Also $w\le v,z$, and each term defining $v$ is at most $w+1-z$, hence $w\ge v+z-1$ and $w\in v\st z$. The interval is symmetric in $x,y,z$, hence also equals $x\st(y\st z)$, proving \eqref{eq:hyperassoc}.
\end{proof}

\begin{corollary}[Finite marginal overlap]\label{cor:finite-overlap}
For $n\ge2$, every bracketing of $x_1\st\cdots\st x_n$ equals the actual attainable spectrum
\begin{equation}
\left\{\mu\!\left(\bigwedge_{i=1}^n a_i\right):\mu(a_i)=x_i\right\}
=\left[\left(\sum_{i=1}^n x_i-(n-1)\right)^+,\min_{1\le i\le n}x_i\right].
\label{eq:finite-overlap}
\end{equation}
\end{corollary}
\begin{proof}
The interval in \eqref{eq:finite-overlap} follows inductively from Theorem~\ref{thm:hyperassoc}. For realizability, suppose the $n$-fold value $v$ is realized by $c=\bigwedge_{i=1}^n a_i$. For any $w\in v\st x_{n+1}$, Proposition~\ref{prop:fixed-realization} supplies $a_{n+1}$ with $\mu(a_{n+1})=x_{n+1}$ and $\mu(c\wedge a_{n+1})=w$. Conversely every actual $(n+1)$-fold intersection has intermediate value $v=\mu(\bigwedge_{i=1}^n a_i)$ and hence belongs to $v\st x_{n+1}$. Induction gives equality; associativity removes bracketing dependence.
\end{proof}

\begin{proposition}[Canonical overlap hypermonoid]\label{prop:hypermonoid}
With set-valued multiplication $\st$, $[0,1]$ is a commutative associative hypermonoid with identity $1$ and absorbing element $0$: $x\st y=y\st x$, $x\st1=\{x\}$, and $x\st0=\{0\}$. The union hyperoperation $\boxp$ is its De Morgan dual under $N(x)=1-x$.
\end{proposition}
\begin{proof}
Commutativity follows from $a\wedge b=b\wedge a$, and associativity is Theorem~\ref{thm:hyperassoc}. By \eqref{eq:frechet}, $x\st1=[x,x]=\{x\}$ and $x\st0=[0,0]=\{0\}$. If $u\in x\boxp y$, choose representatives $a,b$ with $u=\mu(a\vee b)$. De Morgan gives $1-u=\mu(a^c\wedge b^c)\in(1-x)\st(1-y)$. The converse follows by complementing representatives, proving $x\boxp y=1-((1-x)\st(1-y))$.
\end{proof}

The upper endpoint of \eqref{eq:frechet} is $\TG(x,y)=\min\{x,y\}$, precisely Zadeh conjunction. Since Proposition~\ref{prop:complement} has already derived standard complement $N(x)=1-x$, its De Morgan dual is $1-\min\{1-x,1-y\}=\max\{x,y\}$, precisely Zadeh disjunction. Thus the three Zadeh operations arise together from maximal overlap and Boolean complementation. The lower endpoint gives $\TL(x,y)=(x+y-1)^+$, while independent product factors give $\TP(x,y)=xy$ by product-measure multiplicativity. These are the familiar \Luk{} and product alternatives \cite{Klement2000,Hajek1998}; here they correspond to minimum internal overlap and independent-factor composition rather than to the nesting semantics that selects Zadeh conjunction.

\begin{proposition}[Zadeh and alternative measure-derived selections]\label{prop:canonical-selections}
The lower and upper internal overlap selections are $\TL$ and $\TG$, while independent product composition gives $\TP$. Their standard residua are $x\Rightarrow_{\mathrm L}y=\min\{1,1-x+y\}$; $x\Rightarrow_{\mathrm G}y=1$ for $x\le y$ and $y$ otherwise; and $x\Rightarrow_{\mathrm P}y=1$ for $x\le y$ and $y/x$ otherwise. Only the \Luk{} residual negation equals the measure complement $1-x$ for all $x$.
\end{proposition}
\begin{proof}
The first two conjunctions are exactly the endpoints of \eqref{eq:frechet}; the third follows from $\mu_1(a)\mu_2(b)=\mu_1\otimes\mu_2(a\times b)$. A residuum is the greatest $z$ satisfying $T(x,z)\le y$. For $\TL$, $(x+z-1)^+\le y$ is equivalent to $z\le\min\{1,1-x+y\}$. For $\TG$, $\min\{x,z\}\le y$ allows every $z\le1$ when $x\le y$ and exactly $z\le y$ when $x>y$. For product, $xz\le y$ allows every $z\le1$ when $x\le y$ and gives $z\le y/x$ when $x>y$. Setting $y=0$ gives $1-x$ for the \Luk{} residuum; the G\"odel and product residual negations are $1$ at $x=0$ and $0$ for $x>0$.
\end{proof}

\begin{theorem}[Zadeh conjunction as the unique idempotent selection]\label{thm:idempotent}
If $\mu(a)\le\mu(b)$, then $\mu(a\wedge b)=\TG(\mu(a),\mu(b))$ iff $a\le b$. Moreover, if $T(x,y)\in x\st y$ is nondecreasing in each argument and idempotent, then $T=\TG$.
\end{theorem}
\begin{proof}
If $a\le b$, then $a\wedge b=a$. Conversely, if $\mu(a\wedge b)=\mu(a)$, then $\mu(a\setdiff b)=0$, and strict positivity gives $a\le b$. For uniqueness, if $x\le y$, monotonicity and idempotence give $x=T(x,x)\le T(x,y)$, while \eqref{eq:frechet} gives $T(x,y)\le\min\{x,y\}=x$. Thus $T(x,y)=x=\TG(x,y)$; the case $y\le x$ is symmetric.
\end{proof}

\begin{corollary}[Incompatible Boolean laws after scalarization]\label{cor:incompatible-laws}
No nondecreasing overlap selection can simultaneously preserve idempotence $T(x,x)=x$ and complement annihilation $T(x,1-x)=0$ for every $x\in[0,1]$.
\end{corollary}
\begin{proof}
By Theorem~\ref{thm:idempotent}, monotonicity and idempotence force $T=\TG$. But for $0<x<1$, $\TG(x,1-x)=\min\{x,1-x\}>0$.
\end{proof}
This is a direct structural cost of scalarization: the Boolean laws $a\wedge a=a$ and $a\wedge a^c=0$ cannot both be retained by a monotone single-valued selection after the identity/complement relation between representatives has been discarded.

For a chosen intersection selection $T$, Boolean valuation determines the corresponding union value as $U_T(x,y)=x+y-T(x,y)$. This does not by itself imply $U_T(x,y)=1-T(1-x,1-y)$. Requiring equality with the independently defined De Morgan dual is an additional compatibility condition, namely the classical Frank equation
\begin{equation}
T(x,y)+1-T(1-x,1-y)=x+y.
\label{eq:frank}
\end{equation}

\begin{proposition}[Valuation duality of scalar selections]\label{prop:valuation-duality}
If $T(x,y)\in x\st y$, then $U_T(x,y):=x+y-T(x,y)$ is an attainable union selection. Conversely every attainable union selection arises uniquely this way. The correspondence is order reversing. Moreover $U_T$ agrees with the De Morgan dual $1-T(1-x,1-y)$ iff \eqref{eq:frank} holds.
\end{proposition}
\begin{proof}
By \eqref{eq:frechet}, $L\le T\le U$ with $L=(x+y-1)^+$ and $U=\min\{x,y\}$. Subtracting from $x+y$ gives $\max\{x,y\}\le U_T\le\min\{1,x+y\}$, exactly the attainable union interval. The inverse is $T=x+y-U_T$, so the correspondence is bijective and order reversing. The final assertion is algebraically equivalent to \eqref{eq:frank}.
\end{proof}

\begin{proposition}[Residual complement alone does not select \Luk]\label{prop:regraduation}
For $p>0$, let $h_p(x)=x^p/[x^p+(1-x)^p]$ and $T_p(x,y)=h_p^{-1}(\TL(h_p(x),h_p(y)))$. Then $T_p$ is a continuous t-norm and its residual negation is $1-x$. For $p\ne1$, $T_p$ is generally not $\TL$.
\end{proposition}
\begin{proof}
The map $h_p$ is an order automorphism of $[0,1]$ satisfying $h_p(1-x)=1-h_p(x)$. Order conjugation preserves the t-norm axioms and transports residuation. Hence the residual negation of $T_p$ is $h_p^{-1}(1-h_p(x))=1-x$. For $p=2$, $h_2(3/4)=9/10$, so $T_2(3/4,3/4)=h_2^{-1}(4/5)=2/3$, whereas $\TL(3/4,3/4)=1/2$. Thus standard residual negation does not determine conjunction even within the continuous \Luk{} conjugacy class; without continuity there are further examples, notably nilpotent minimum \cite{Fodor1995}.
\end{proof}

\begin{theorem}[Measure-compatible residuated selection]\label{thm:measure-compatible}
If $T$ is residuated, $x\Rightarrow_T0=1-x$, and \eqref{eq:frank} holds, then $T=\TL$.
\end{theorem}
\begin{proof}
If $x+y\le1$, then $y\le1-x=x\Rightarrow_T0$, so residuation gives $T(x,y)=0$. If $x+y\ge1$, then $(1-x)+(1-y)\le1$, so the first case gives $T(1-x,1-y)=0$; equation~\eqref{eq:frank} yields $T(x,y)=x+y-1$. Hence $T(x,y)=(x+y-1)^+$.
\end{proof}

The product selection has a second measure-theoretic characterization stronger than the familiar sandwich $\TL\le\TP\le\TG$.

For the remainder of this section, write $L(x,y)=(x+y-1)^+$ and $U(x,y)=\min\{x,y\}$ for the two endpoints in \eqref{eq:frechet}.

\begin{proposition}[Pointwise minimax center]\label{prop:pointwise-minimax}
Fix marginals $(x,y)$. The unique scalar $r$ minimizing $\sup_{z\in[L(x,y),U(x,y)]}|z-r|$ is $[L(x,y)+U(x,y)]/2$, with minimum error $[U(x,y)-L(x,y)]/2=\min\{x,y,1-x,1-y\}/2$.
\end{proposition}
\begin{proof}
For any $r$, the worst error on the interval is $\max\{r-L(x,y),U(x,y)-r\}$, with the same conclusion if $r$ lies outside. This is uniquely minimized when the two endpoint errors are equal.
\end{proof}

\begin{theorem}[Sharp minimax scalar prediction]\label{thm:binary-minimax}
For any $F:[0,1]^2\to[0,1]$, let $\cE(F)=\sup_{a,b\in\cB}|\mu(a\wedge b)-F(\mu(a),\mu(b))|$. Then $\inf_F\cE(F)=1/4$. The same optimum holds when $F$ is restricted to t-norms, and $F(x,y)=xy$ attains it.
\end{theorem}
\begin{proof}
At $x=y=1/2$, \eqref{eq:frechet} allows intersection values $0$ and $1/2$, so every prediction has error at least $1/4$. For $F(x,y)=xy$, every attainable $z$ lies in $[L(x,y),U(x,y)]$, so the worst error at $(x,y)$ is $\max\{xy-L(x,y),U(x,y)-xy\}$. If $x+y\le1$, then $xy-L(x,y)=xy\le1/4$; otherwise $xy-L(x,y)=(1-x)(1-y)\le1/4$. Assuming $x\le y$, $U(x,y)-xy=x(1-y)\le x(1-x)\le1/4$; the other case is symmetric. Thus $\cE(\TP)\le1/4$, and equality follows from the lower bound.
\end{proof}

\begin{theorem}[All globally minimax scalar predictors]\label{thm:all-minimax}
A predictor $F:[0,1]^2\to[0,1]$ satisfies $\cE(F)=1/4$ iff
$U(x,y)-1/4\le F(x,y)\le L(x,y)+1/4$ for every $(x,y)$.
\end{theorem}
\begin{proof}
For fixed $(x,y)$ every attainable intersection lies in $[L(x,y),U(x,y)]$, and both endpoints are attainable. Hence the worst error is exactly $\max\{|F(x,y)-L(x,y)|,|F(x,y)-U(x,y)|\}$. This is at most $1/4$ iff $F(x,y)\le L(x,y)+1/4$ and $F(x,y)\ge U(x,y)-1/4$. These constraints are compatible because $U(x,y)-L(x,y)\le1/2$. Imposing them for every pair is therefore equivalent to $\cE(F)\le1/4$, while Theorem~\ref{thm:binary-minimax} rules out any smaller global error.
\end{proof}

Thus product is distinguished by independence and is minimax-optimal, but minimaxity alone does not characterize it. The pointwise minimax predictor in Proposition~\ref{prop:pointwise-minimax} is the midpoint of the admissible interval, whereas product chooses the zero-dependence point of Proposition~\ref{prop:dependence}; these coincide only on special marginals.

\begin{proposition}[Pairwise optimality is not associative]\label{prop:midpoint-nonassoc}
Let $C(x,y)=\frac12[L(x,y)+U(x,y)]$ be the unique pointwise minimax predictor of Proposition~\ref{prop:pointwise-minimax}. Then $C$ is not associative and therefore is not a t-norm.
\end{proposition}
\begin{proof}
One has $C(1/4,1/2)=1/8$ and $C(1/2,1/2)=1/4$. Hence $C(C(1/4,1/2),1/2)=1/16$, whereas $C(1/4,C(1/2,1/2))=1/8$.
\end{proof}

\begin{theorem}[Quantitative cost of the canonical Boolean remnants]\label{thm:canonical-cost}
For the three canonical scalar conjunctions, $\cE(\TL)=\cE(\TG)=1/2$ and $\cE(\TP)=1/4$. Consequently every nondecreasing idempotent overlap selection has worst-case error $1/2$.
\end{theorem}
\begin{proof}
At fixed marginals, predicting the lower endpoint $L(x,y)$ makes the upper endpoint the worst case, while predicting $U(x,y)$ makes the lower endpoint the worst case. Thus both errors equal $U(x,y)-L(x,y)=\min\{x,y,1-x,1-y\}\le1/2$, and equality is attained at $x=y=1/2$. Product has sharp error $1/4$ by Theorem~\ref{thm:binary-minimax}. Finally Theorem~\ref{thm:idempotent} identifies every nondecreasing idempotent overlap selection with $\TG$, so each has the same $1/2$ worst-case error.
\end{proof}

\begin{theorem}[Sharp $n$-ary minimax prediction]\label{thm:nary-minimax}
For $n\ge2$ and $F:[0,1]^n\to[0,1]$, define
$\cE_n(F)=\sup_{a_1,\ldots,a_n}|\mu(\bigwedge_{i=1}^n a_i)-F(\mu(a_1),\ldots,\mu(a_n))|$. Then $\inf_F\cE_n(F)=(n-1)/(2n)$. For fixed marginals the unique minimax predictor is the midpoint of the finite overlap interval in Corollary~\ref{cor:finite-overlap}.
\end{theorem}
\begin{proof}
Put $L_n=(\sum_{i=1}^n x_i-(n-1))^+$ and $U_n=\min_{1\le i\le n}x_i$. Corollary~\ref{cor:finite-overlap} makes the attainable intersection exactly $[L_n,U_n]$, so the pointwise minimax error is $(U_n-L_n)/2$. Let $m=U_n$. If $L_n=0$, then $nm\le\sum_{i=1}^n x_i\le n-1$, hence $U_n-L_n=m\le(n-1)/n$. If $L_n>0$, then $U_n-L_n=m-\sum_{i=1}^n x_i+n-1\le(n-1)(1-m)$, while trivially $U_n-L_n\le m$; hence it is at most $\min\{m,(n-1)(1-m)\}\le(n-1)/n$. Equality holds at $x_1=\cdots=x_n=(n-1)/n$.
\end{proof}

\begin{theorem}[Exact $n$-ary product error]\label{thm:nary-product}
For $P_n(x_1,\ldots,x_n)=\prod_{i=1}^n x_i$ and $n\ge2$,
$\cE_n(P_n)=(n-1)n^{-n/(n-1)}$. Hence product is minimax among arities $n\ge2$ exactly when $n=2$; its worst-case error tends to $1$, whereas the unrestricted minimax error in Theorem~\ref{thm:nary-minimax} tends to $1/2$.
\end{theorem}
\begin{proof}
Let $m=\min_{1\le i\le n}x_i$, $s=\sum_{i=1}^n x_i$, $L=(s-n+1)^+$, $U=m$, and $\Pi=\prod_{i=1}^n x_i$. Since $1-\Pi\le\sum_{i=1}^n(1-x_i)$, one has $\Pi\ge L$, while $\Pi\le U$. Thus the error is $\max\{U-\Pi,\Pi-L\}$. Because every $x_i\ge m$, $U-\Pi\le m-m^n$, with equality when all $x_i=m$; maximizing gives $(n-1)n^{-n/(n-1)}$. By AM--GM, $\Pi-L\le(s/n)^n-(s-n+1)^+$, whose supremum is $((n-1)/n)^n$. With $q=(n-1)/n$ and $q^{n-1}\le1/2$, the first supremum dominates; for $n>2$ it is strictly larger than the minimax value $q/2$. The case $n=2$ is Theorem~\ref{thm:binary-minimax}.
\end{proof}

\section{Ordered-Valued Measures}\label{sec:ordered}
Let $(G,u)$ be a unital Abelian $\ell$-group and $\Gamma(G,u)=[0,u]_G$ \cite{Mundici1986,Fuchs1963,Glass1999}. Let $\mu:\cB\to G$ be positive, finitely additive on disjoint joins, strictly positive, normalized by $\mu(1)=u$, and interval divisible: whenever $0\le g\le\mu(a)$ there is $b\le a$ with $\mu(b)=g$. Define \eqref{eq:star-def} for $x,y\in\Gamma(G,u)$, and write $[r,s]_G=\{g\in G:r\le g\le s\}$. In this section $\wedge,\vee$ on $G$ denote lattice operations; on $\cB$ they retain their Boolean meaning.

\begin{theorem}[Ordered overlap spectrum]\label{thm:ordered-spectrum}
For $x,y\in\Gamma(G,u)$,
$x\st y=[(x+y-u)\vee0,\ x\wedge y]_G$.
\end{theorem}
\begin{proof}
If $g=\mu(a\wedge b)$, positivity gives $g\le x\wedge y$. Also $x+y-g=\mu(a\vee b)\le u$, so $g\ge x+y-u$, and $g\ge0$. Conversely let $(x+y-u)\vee0\le g\le x\wedge y$. Then $g$, $x-g$, $y-g$, and $u-x-y+g$ are positive and sum to $u$. Repeated interval divisibility produces disjoint pieces $e_{11},e_{10},e_{01},e_{00}$ with these measures. Setting $a=e_{11}\vee e_{10}$ and $b=e_{11}\vee e_{01}$ gives the claim.
\end{proof}

\begin{corollary}[Ordered extremal operations]\label{cor:ordered-extremal}
The least attainable intersection is $x\odot y=(x+y-u)\vee0$, the greatest attainable union is $x\oplus y=(x+y)\wedge u$, and Boolean complement induces $\neg x=u-x$. These are exactly the standard operations of $\Gamma(G,u)$.
\end{corollary}
\begin{proof}
The least intersection is the lower endpoint of Theorem~\ref{thm:ordered-spectrum}. For representatives of values $x,y$ with intersection value $g$, finite additivity gives $\mu(a\vee b)=x+y-g$. Hence the largest attainable union is obtained at the least admissible $g=(x+y-u)\vee0$ and equals $x+y-[(x+y-u)\vee0]=(x+y)\wedge u$. Finally $u=\mu(a)+\mu(a^c)$ gives $\mu(a^c)=u-\mu(a)$.
\end{proof}
The formulas are classical in $\Gamma(G,u)$ \cite{Mundici1986,Cignoli2000}; their derivation as extremal attainable overlaps is the point.

\begin{theorem}[Ordered hyperassociativity]\label{thm:ordered-hyperassoc}
For $x,y,z\in\Gamma(G,u)$,
$(x\st y)\st z=x\st(y\st z)=[(x+y+z-2u)\vee0,\ x\wedge y\wedge z]_G$.
\end{theorem}
\begin{proof}
If $v\in x\st y$ and $w\in v\st z$, then $w\le x\wedge y\wedge z$, $w\ge0$, and $w\ge v+z-u\ge x+y+z-2u$. Conversely take $w$ in the stated interval, put $L=(x+y-u)\vee0$, and $v=L\vee w$. Since $L,w\le x\wedge y$, Theorem~\ref{thm:ordered-spectrum} gives $v\in x\st y$. Also $w\le v\wedge z$. Translation preserves joins, so
$v+z-u=(x+y+z-2u)\vee(z-u)\vee(w+z-u)\le w$; hence $w\ge(v+z-u)\vee0$, and Theorem~\ref{thm:ordered-spectrum} gives $w\in v\st z$. Symmetry proves associativity.
\end{proof}

\begin{corollary}[Finite ordered overlap]\label{cor:ordered-finite}
For $n\ge2$, every bracketing of $x_1\st\cdots\st x_n$ equals the actual attainable spectrum
\begin{multline}
\left\{\mu\!\left(\bigwedge_{i=1}^n a_i\right):\mu(a_i)=x_i\right\}\\
=\left[\left(\sum_{i=1}^n x_i-(n-1)u\right)\vee0,\ \bigwedge_{i=1}^n x_i\right]_G.
\label{eq:ordered-finite}
\end{multline}
\end{corollary}
\begin{proof}
The interval in \eqref{eq:ordered-finite} follows by induction from Theorem~\ref{thm:ordered-hyperassoc}. For realizability, suppose $v=\mu(c)$ is realized by $c=\bigwedge_{i=1}^n a_i$ and take $w\in v\st x_{n+1}$. Theorem~\ref{thm:ordered-spectrum} gives $0\le w\le v$ and $0\le x_{n+1}-w\le u-v=\mu(c^c)$. Interval divisibility yields disjoint $p\le c$, $q\le c^c$ with measures $w$ and $x_{n+1}-w$; then $a_{n+1}=p\vee q$ realizes $w$. The converse is immediate from the intermediate intersection value. Induction identifies the hyperoperation with one consistent event family.
\end{proof}

\begin{corollary}[Finite ordered union spectrum]\label{cor:ordered-union}
For $n\ge2$, the attainable values of $\mu(\bigvee_{i=1}^n a_i)$ under $\mu(a_i)=x_i$ form
$[\bigvee_{i=1}^n x_i,(\sum_{i=1}^n x_i)\wedge u]_G$.
\end{corollary}
\begin{proof}
By De Morgan, $\mu(\bigvee_{i=1}^n a_i)=u-\mu(\bigwedge_{i=1}^n a_i^c)$ and $\mu(a_i^c)=u-x_i$. Applying \eqref{eq:ordered-finite} to the complements gives the intersection interval $[(u-\sum_{i=1}^n x_i)\vee0,\,u-\bigvee_{i=1}^n x_i]_G$. Subtracting this interval from $u$ reverses its endpoints and yields $[\bigvee_{i=1}^n x_i,(\sum_{i=1}^n x_i)\wedge u]_G$.
\end{proof}

\section{Local Versus Global Realization}\label{sec:global}
Equation~\eqref{eq:frechet} is a pointwise marginal statement. Simultaneous realization over all thresholds is stronger and is governed in two dimensions by copulas \cite{Sklar1959,Nelsen2006}.

\begin{proposition}[Canonical bivariate global realizations]\label{prop:bivariate-global}
The G\"odel, \Luk, and product selections admit simultaneous realization for every $(x,y)$ by comonotone, countermonotone, and independent event chains, respectively.
\end{proposition}
\begin{proof}
On $([0,1],\lambda)$ let $A_x=[0,x]$. The comonotone chain $B_y=[0,y]$ gives $\lambda(A_x\cap B_y)=\min\{x,y\}=\TG(x,y)$. The countermonotone chain $B_y=[1-y,1]$ gives intersection length $\max\{0,x+y-1\}=\TL(x,y)$. For product, on $([0,1]^2,\lambda_2)$ take $A_x=[0,x]\times[0,1]$ and $B_y=[0,1]\times[0,y]$; independence of the coordinate strips gives $\lambda_2(A_x\cap B_y)=xy=\TP(x,y)$ for every pair of thresholds simultaneously.
\end{proof}
The comonotone construction is the global measure model of Zadeh conjunction: one nested chain realizes $\min\{x,y\}$ for every pair of degrees at once. Applying complements to the same chain realizes the associated maximum disjunction. This strengthens pointwise endpoint selection to a simultaneous semantics for the full Zadeh min--max calculus.
Hyperassociativity, however, does not promote the lower Fr\'echet endpoint to a multivariate copula.

\begin{proposition}[Three-dimensional obstruction]\label{prop:three-dim}
The iterated lower endpoint $W_3(x,y,z)=(x+y+z-2)^+$ cannot be realized as $\mu(A_x\wedge B_y\wedge C_z)$ for three increasing event chains with $\mu(A_x)=x$, $\mu(B_y)=y$, and $\mu(C_z)=z$ simultaneously for all $(x,y,z)\in[0,1]^3$.
\end{proposition}
\begin{proof}
For any such chains, the alternating volume over $[r,1]^3$ equals $\mu((A_1\setdiff A_r)\wedge(B_1\setdiff B_r)\wedge(C_1\setdiff C_r))$ and is nonnegative. Taking $r=1/2$ and $W_3$ gives $1-3(1/2)+3(0)-0=-1/2$, a contradiction.
\end{proof}
This is the familiar dimension-dependent Fr\'echet/copula obstruction \cite{McNeil2009}, stated here to delimit the hyperalgebra correctly: associativity and pointwise sharp attainability describe unconstrained marginal composition, not preservation of all higher-order dependencies. Likewise, \eqref{eq:star-def} should not be recursively applied to repeated variables as if scalarization retained their identity.

\begin{corollary}[Dimension split of the canonical selections]\label{cor:dimension-split}
For every $n$, $G_n(x_1,\ldots,x_n)=\min_{1\le i\le n}x_i$ and $P_n(x_1,\ldots,x_n)=\prod_{i=1}^n x_i$ admit simultaneous realizations by $n$ increasing event chains. The lower endpoint $W_n=(\sum_{i=1}^n x_i-(n-1))^+$ does not admit such a realization for any $n\ge3$.
\end{corollary}
\begin{proof}
For $G_n$, on $([0,1],\lambda)$ take $A^{(i)}_{x_i}=[0,x_i]$. For $P_n$, on $([0,1]^n,\lambda_n)$ take $A^{(i)}_{x_i}=\{t:t_i\le x_i\}$. If $W_n$ were simultaneously realizable for some $n\ge3$, fixing $x_4=\cdots=x_n=1$ would realize $W_3$, contradicting Proposition~\ref{prop:three-dim}.
\end{proof}

\begin{proposition}[T-norm admissibility is weaker than measure admissibility]\label{prop:tnorm-weaker}
There exist continuous t-norms $T$ for which $T(x,y)\notin x\st y$ at some marginals.
\end{proposition}
\begin{proof}
Conjugate $\TL$ by the order automorphism $h(x)=x^2$: $T_h(x,y)=h^{-1}(\TL(h(x),h(y)))=\sqrt{\max\{0,x^2+y^2-1\}}$. Order conjugation preserves commutativity, associativity, monotonicity, continuity, and neutral element $1$, so $T_h$ is a continuous t-norm. But $T_h(3/4,3/4)=1/(2\sqrt2)<1/2=\TL(3/4,3/4)$, violating the necessary lower bound in \eqref{eq:frechet}.
\end{proof}

Once the Zadeh scalar calculus has been derived, ordinary fuzzy-set semantics is downstream: $1-x$, $\min$, and $\max$ lift pointwise to $[0,1]^X$, characteristic functions recover the crisp Boolean core, and measurable functions remain closed under these continuous operations. The usual functional representation of Zadeh fuzzy sets is therefore the endpoint of the derivation, not the premise from which the operations are assumed.

\section{Conclusion}\label{sec:conclusion}
Zadeh fuzzy logic follows from the measure-algebraic construction in three steps. Boolean complement descends uniquely to $1-x$. Boolean meet descends exactly to the associative overlap interval $[(x+y-1)^+,\min\{x,y\}]$, rather than to a single scalar. Finally, requiring the scalar conjunction to be nondecreasing and idempotent---equivalently, preserving the nesting behavior of events---selects $\min\{x,y\}$ uniquely. De Morgan duality then gives $\max\{x,y\}$, and a nested event chain realizes the min--max pair simultaneously. The standard complement, minimum conjunction, and maximum disjunction of Zadeh fuzzy logic are therefore derived rather than postulated.

The same construction also makes the price of scalarization explicit. No nontrivial measure-preserving quotient retains Boolean meet, and the discrepancy formula quantifies the distinctions erased inside each scalar fiber. Minimum overlap and independence yield the \Luk{} and product conjunctions as alternative semantic selections. Product has sharp binary minimax error $1/4$, although the corresponding $n$-ary operation ceases to be minimax beyond two arguments. Interval-divisible measures in unital Abelian $\ell$-groups retain the exact finite overlap spectrum and hyperassociativity.

The three-dimensional obstruction marks the boundary of the derivation: pointwise sharp attainability does not by itself provide simultaneous realization of all marginal combinations. Zadeh's nesting selection avoids that obstruction and admits a global chain model in every finite arity. Measure theory therefore supplies both a derivation of the Zadeh calculus and a precise account of the additional dependence choice encoded by its min--max operations.

\bibliography{references}
\end{document}